\documentclass[12pt]{amsart}

\usepackage{graphicx}
\usepackage{stmaryrd}
\usepackage{amssymb}

\usepackage{amsmath,amscd}

\usepackage[T1]{fontenc}
\usepackage{yfonts}

\newcommand{\har}{\!\!\upharpoonright\!\!}
\renewcommand{\int}{\mbox{int}}

\newcommand{\cl}{\mathrm{cl}}

\newcommand{\subsetnwd}{\subseteq_{\mathrm{nwd}}}
\newcommand{\gdelta}{\mathbf{G}_\delta}
\newcommand{\btree}{\omega^{<\omega}}
\newcommand{\lh}{\mathrm{lh}}
\newcommand{\proj}{\mathrm{proj}}
\newcommand{\baire}{\omega^\omega}
\newcommand{\cantor}{2^\omega}

\newtheorem{theorem}{Theorem}
\newtheorem{lemma}{Lemma}

\theoremstyle{definition}

\newtheorem*{claim}{Claim}

\author{Janusz Pawlikowski}
\author{Marcin Sabok}
\thanks{research supported by MNiSW grant N 201 361836} 
\address{Mathematical Institute,
  Wroc\l aw University, pl. Grunwaldzki $2\slash 4$,
  $50$-$384$ Wroc\l aw, Poland }
\email{pawlikow@math.uni.wroc.pl, sabok@math.uni.wroc.pl}

\title{The Solecki dichotomy for functions with analytic
  graphs}

\begin{document}

\begin{abstract}
  A dichotomy discovered by Solecki says that a Baire class
  1 function from a Souslin space into a Polish space either
  can be decomposed into countably many continuous
  functions, or else contains one particular function which
  cannot be so decomposed. In this paper we generalize this
  dichotomy to arbitrary functions with analytic graphs. We
  provide a ``classical'' proof, which uses only elementary
  combinatorics and topology.
\end{abstract}



\maketitle

\section{Introduction}

An old question of Lusin asked whether there exists a Borel
function which cannot be decomposed into countably many
continuous functions. By now, several examples have been
given, by Keldi\v{s}, Adyan and Novikov among others. A
particularly simple example, the function $P$ has been found
in \cite{CMPS}. To define it, we introduce two topologies on
$\baire$.  The first one, referred to as \textit{the Baire
  topology} is the usual one.  The second, referred to as
\textit{the Cantor topology} is the one induced by the
identification of $\omega$ with
$\{0\}\cup\{2^{-n}:n<\omega\}\subseteq[0,1]$ via $0\mapsto0,
n\mapsto2^{-n}$. Note that the space $\baire$ with the
Cantor topology is homeomorphic to the Cantor space
$\cantor$. The function $P$ is the identity function from
$\baire$ with the Cantor topology into $\baire$ with the
Baire topology.

Let $H=\dot H=\ddot H$ be the Hilbert cube. We consider
subsets $\dot A\subseteq \dot H$, $\ddot A\subseteq\ddot H$
and a function $A:\dot A\rightarrow\ddot A$. In this paper
we will prove the following result.

\begin{theorem}\label{pmin}
  If $\dot A$ and $\ddot A$ are analytic and $A:\dot
  A\rightarrow\ddot A$ is analytic as a subset of $\dot
  A\times\ddot A$, then either $A$ can be covered by
  countably many continuous functions, or else there are
  $\dot\varphi:\baire\rightarrow \dot A$ and
  $\ddot\varphi:\baire\rightarrow \ddot A$ such that the
  following diagram commutes 
  $$
  \begin{CD}
    \baire @ >\ddot\varphi >> \ddot A\\
    @AA P A @AA A A\\
    \baire @> \dot\varphi >> \dot A
  \end{CD}
  $$
  and
  \begin{itemize}
  \item $\dot\varphi$ is a topological embedding from
    $\baire$ with the Cantor topology into $\dot A$,
  \item $\ddot\varphi$ is a topological embedding from
    $\baire$ with the Baire topology into $\ddot A$.
  \end{itemize}
\end{theorem}

Theorem \ref{pmin} has been proved by Solecki \cite[Theorem
4.1]{Sol} in the case when $A$ is a Baire class 1 function.
Zapletal \cite[Corollary 2.3.48]{Zpl:DSTDF} proved this
dichotomy in the case when $A$ is a Borel function and $\dot
A=\baire$ (however, that argument does not seem to work in
case when $\dot A$ is an arbitrary Souslin space). The
formulation of Theorem \ref{pmin} seems to be a natural
generalization of both these cases.

It should be noted that both proofs of Solecki and Zapletal
use quite sophisticated methods of mathematical logic and
the Borel case uses the Baire class 1 case and Borel
determinacy. On the other hand, our proof works in the
second-order arithmetic.

\section{Notation}

If $T\subseteq S^{<\omega}$ ($S$ is some countable set) is a
tree, then we write $\lim T$ for $\{s\in S^\omega: \forall
n<\omega\ s\har n\in T\}$. For $\sigma\in S^{<\omega}$
we denote by $[\sigma]$ the set $\{s\in
S^\omega:\ \sigma\subseteq s\}$.

The Hilbert cube $H$ is endowed with the standard complete
metric. We write $|x,y|$ for the distance between points
$x,y\in H$, $|X|$ for the diameter of a subset $X\subseteq
H$ and $|X,Y|$ for the Hausdorff distance between
$X,Y\subseteq H$.

For $k\leq l<\omega$ write
\begin{displaymath}
|k,l|= \left\{ 
    \begin{array}{ll}
      2^{-k}-2^{-l} & \mbox{if}\quad k>0,\\
      2^{-l} & \mbox{if}\quad k=0
    \end{array} \right.   
\end{displaymath}
and for $l<k<\omega$ let  $|k,l|=|l,k|$.

If $0<N\leq\omega$ and $s,r\in\omega^N$, define
$$|s,r|=\sum_{n<N}|s(n),r(n)|\cdot2^{-n}.$$ Notice that for
$N=\omega$ the above metric gives the Cantor topology on
$\baire$.

We fix also a complete metric for $(\baire)^\omega$ and use
$|\cdot,\cdot|$ to denote distance and diameter.

For $(r,\dot x,\ddot x),(s,\dot y,\ddot
y)\in(\baire)^\omega\times\dot H\times\ddot H$ write
$$|(r,\dot x,\ddot x),(s,\dot y,\ddot y)|=\max(|r,s|,|\dot
x,\dot y|,|\ddot x,\ddot y|)$$ and, as above, write $|\tilde
X|$ for the diameter of $\tilde
X\subseteq(\baire)^\omega\times(\dot H\times\ddot H)$

Let $\pi:(\baire)^\omega\times(\dot H\times\ddot
H)\rightarrow\dot H\times\ddot H$,
$\dot\pi:(\baire)^\omega\times(\dot H\times\ddot
H)\rightarrow\dot H$ and
$\ddot\pi:(\baire)^\omega\times(\dot H\times\ddot
H)\rightarrow\ddot H$ be the projection maps.

If $X$ is a subset of $A$, then by $\dot X$ and $\ddot X$ we
denote the projections of $X$ onto $\dot H$ and $\ddot H$,
respectively. We write $Y\subsetnwd X$ if $Y\subseteq X$ and
$\dot Y\subseteq\cl(\dot X\setminus \dot Y)$.

\section{The dichotomy}

In the remaining part of this paper we give a proof of Theorem
\ref{pmin}.

We denote by $I$ the $\sigma$-ideal of subsets of $A$
generated by graphs of continuous functions. We say that a
subset $X$ of $A$ is \textit{$I$-positive} if $X\not\in I$.

\begin{lemma}\label{lemma1}
  If $X\subseteq A$ is $I$-positive, then
  \begin{enumerate}
  \item $\ddot X$ is uncountable,
  \item there exists $I$-positive $Y\subseteq X$ such that
    $Y$ is relatively closed in $X$ and $Y\subsetnwd X$.
  \end{enumerate}
\end{lemma}
\begin{proof}
  The first part is obvious. Note that $X:\dot
  X\rightarrow\ddot X$ is a Baire measurable function, so
  $\dot X$ has a dense $\gdelta$ (relative) subset $G$ such
  that $X\har G$ is continuous. Write $\dot X\setminus
  G=\bigcup_{n<\omega}F_n$, where each $F_n$ is closed
  nowhere dense (relative) subset of $\dot X$. Now
  $$X=X\har G\quad\cup\quad\bigcup_{n<\omega} X\har F_n.$$
  Since $X\har G\in I$, at least one $X\har F_n$ is
  $I$-positive. Note that $X\har F_n=X\cap(F_n\times\ddot
  H)$ is relatively closed in $X$.
\end{proof}

\begin{lemma}\label{lemma2}
  Let $\varepsilon>0$ and $\dot X\subseteq\dot H$. There is
  $K\in\omega$ such that for any sequence $\langle\dot
  X_k\subseteq\dot X:k<K\rangle$ such that for $k\not=k'<K$
  $$\dot X_k\cup \dot X_{k'}\mbox{ is dense in }\dot X$$
  there is $k<K$ such that $$|\dot X_k,\dot
  X|<\varepsilon.$$
\end{lemma}

\begin{proof}

  Using compactness of $\dot H$, find finite set $\dot
  Y\subseteq \dot X$ such that $|\dot Y,\dot X|<\varepsilon
  $. Let $K$ be any number greater than the cardinality of
  $\dot Y$.

  Since $|\dot X_k,\cl\dot X_k|=0$, without loss of
  generality assume that all $\dot X_k$ are closed and for
  $k\not=k'$ we have $\dot X_k\cup\dot X_{k'}=\dot X$. We
  search for $k$ such that $\dot Y\subseteq \dot X_k$. If
  every $k$ fails at some point of $\dot Y$, then there are
  $k\not=k'$ that fail at the same point. This violates
  $\dot X_k\cup\dot X_{k'}=\dot X$.
\end{proof}

\subsection{Projection and ordering}

For $\sigma\in\btree$ write $\lh\sigma$ for the length of
$\sigma$ and $\sigma^*$ for $\sigma$ with the last digit
removed if $\sigma\not=\emptyset$ (and let
$\emptyset^*=\emptyset$).

For $n<\omega$ we define the projection function
$$\omega^n\ni\sigma\mapsto\sigma'\in\omega^n$$ as
follows: the first largest digit of $\sigma$ which is $\geq
n$ is changed to $0$; if there is no such digit, we put
$\sigma'=\sigma$.

Write $\sigma^1=\sigma'$ and $\sigma^{i+1}=(\sigma^i)'$.
Note that for $\sigma\in\omega^n$
\begin{itemize}
\item $(\sigma^n)'=\sigma^n$,
\item $|\sigma,\sigma'|\leq2^{-n}\cdot2^{-i}$ if the change
  occures at $i$-th digit,
\item
  $|\sigma,\sigma''|=|\sigma,\sigma'|+|\sigma',\sigma''|$,
  since the consecutive changes occure at different places,
\item $|\sigma,\sigma^n|\leq2\cdot2^{-n}$
\end{itemize}
Fix an ordering $\preceq$ of $\btree$ into type $\omega$
such that $\sigma'\preceq\sigma$ and $\sigma^*\preceq\sigma$
for each $\sigma\in\btree$. Write $\#\sigma$ for the number
indicating the position of $\sigma$ with respect to
$\preceq$. Note that $\#\emptyset=0$, $\#\langle0\rangle=1$
and $\lh\sigma\leq\#\sigma$

\subsection{Solids}

We call a nonempty analytic $X\subseteq A$ \textit{solid} if
for all open $U\subseteq \dot H\times \ddot H$ either $U\cap
X=\emptyset$, or else $U\cap X$ is $I$-positive. Note that
every analytic $I$-positive set contains a solid. If $X$ is
a solid, then for each open set $U\subseteq \dot
H\times\ddot H$ either $U\cap X=\emptyset$, or else $U\cap
X$ is solid.

Without loss of generality we assume that \textbf{$A$ is solid}.

\subsection{Trees}

For a tree $T$ on $\omega\times\omega$ write
$\proj[T]=\{\sigma\in\btree:\exists\tau\ (\sigma,\tau)\in
T\}$.  For $\sigma\in\btree$ write
$T_\sigma=\{\tau:(\sigma,\tau)\in T\}$ and for $s\in\baire$
write $T_s=\bigcup_{n<\omega} T_{s\upharpoonright n}$. If
$\{X_{\rho,\tau}: (\rho,\tau)\in T\}$ is a family of sets
and $\sigma\in \btree$, put
$$X_\sigma=\bigcup\{X_{\sigma,\tau}:\tau\in T_\sigma\}.$$

\begin{lemma}\label{lemma3}
  Suppose $T$ is a tree on $\omega\times\omega$ such that
  for each $s\in\baire$ the tree $T_s$ is finitely branching
  and there is exactly one branch $\psi(s)\in\lim T_s$. Then
  the map $\psi:\baire\rightarrow\baire$ is continuous.
\end{lemma}

\begin{proof}
  Let $t=\psi(s)$ and fix $n<\omega$. We need to find $m\geq
  n$ such that $$\psi\big[[s\har m]\big]\subseteq[t\har
  n].$$ Suppose towards a contradiction that for each $m\geq
  n$ there is $s_m\supseteq s\har m$ such that
  $\psi(s_m)\not\supseteq t\har n$. Consider the tree
  $$T^*=\{\tau\in T_s:\tau\not\supseteq t\har n\}.$$ Then
  for each $m\geq n$ we have
  \begin{itemize}
  \item $\psi(s_m)\har m\in T_{s_m\upharpoonright
      m}=T_{s\upharpoonright m}\subseteq T_s$,
  \item $\psi(s_m)\har m\not\supseteq t\har n$, i.e.
    $\psi(s_m)\har m\in T^*$.
  \end{itemize}
  So $T^*$ is an infinite finitely branching tree. Pick
  $t^*\in\lim T^*$. Then $t^*\in\lim T_s$ but $t^*\not=t$, a
  contradiction.
\end{proof}

\subsection{Cylinders}

By a \textit{cylinder with base $X$} we mean a closed set
$\tilde X\subseteq(\baire)^\omega\times \dot H\times\ddot H$
such that $\pi[\tilde X]=X$ and there exists $N<\omega$ such
that for each $x\in X$ and for each $r,s\in(\baire)^\omega$
  $$\big( s\har N=r\har N\
  \wedge\ (s,x)\in\tilde X\big)\Rightarrow (r,x)\in\tilde
  X.$$ In this case we say that \textit{$X$ is unfolded to
    $\tilde X$}.

Note that the base of a cylinder is analytic and that every
analytic subset of $A$ can be unfolded to a cylinder. A
cylinder is \textit{solid} if its base is solid. For the
rest of the proof fix a solid cylinder $\tilde A$ with base
$A$.

\begin{lemma}\mbox{}
  \begin{itemize}
  \item[(a)] Given a cylinder $\tilde X$ with base $X$, any
    analytic $Y\subseteq X$ can be unfolded to a cylinder
    $\tilde Y\subseteq \tilde X$.
  \item[(b)] Given $\varepsilon>0$, solid cylinder $\tilde
    X$ and an analytic $I$-positive $Y\subseteq X$, there is
    a solid $Z\subsetnwd Y$ that can be unfolded to a
    cylinder $\tilde Z\subseteq\tilde X$ such that $|\tilde
    Z|<\varepsilon$.
  \end{itemize}
\end{lemma}

\begin{proof}
  (a) Fix closed $D\subseteq\baire\times \dot H\times\ddot
  H$ that projects onto $Y$. Let $N<\omega$ witness that
  $\tilde X$ is a cylinder. Define $\tilde Y$ by
  $$(x,z)\in\tilde Y\quad\mbox{iff}\quad (s,z)\in\tilde X\
  \wedge\ (s(N),z)\in D.$$

  (b) By Lemma \ref{lemma1} find analytic $I$-positive
  $Y_0\subsetnwd Y$. Unfold $Y_0$ to a cylinder $\tilde
  Y_0\subseteq\tilde X$ and write $\tilde Y_0$ as a
  countable union of cylinders of diameter less than
  $\varepsilon$. At least one of them, say $\tilde Y$, has
  $I$-positive base. Shrink this base to a solid $Z$ and
  unfold $Z$ to a solid cylinder $\tilde Z\subseteq \tilde
  Y$.
\end{proof}

\subsection{Solid trees}

Given a finite tree $S$ on $\omega\times\omega$, call a
family $$\langle \tilde X_{\sigma,\tau}:\ (\sigma,\tau)\in
S\rangle$$ of solid cylinders a \textit{solid tree} if
\begin{itemize}
\item $\dot X_{\sigma,\tau}\subseteq\cl\dot X_{\sigma^*,\tau^*}$,
\item $\ddot X_{\sigma,\tau}$ are pairwise disjoint
  and relatively open in $\bigcup_{(\sigma,\tau)\in S}\ddot
  X_{\sigma,\tau}$.
\end{itemize}

\begin{lemma}\label{lemma5}
  Let $\varepsilon>0$ and $\langle\tilde
  X_{\sigma,\tau}:(\sigma,\tau)\in S\rangle$ be a solid
  tree. Suppose $\eta\not\in\proj[S]$ and
  $\eta^*\in\proj[S]$. Let $Y', Y\subseteq \dot H\times
  \ddot H$ be such that $\dot Y'\subseteq \cl\dot Y$. Then
  there exists a solid tree $\langle\tilde
  X'_{\sigma,\tau}:(\sigma,\tau)\in S'\rangle$ such that
  $S'\supseteq S$, $\proj[S']=\proj[S]\cup\{\eta\}$ and
  \begin{enumerate}
  \item for $\tau\in S'_\eta$ we have $|\tilde
    X'_{\eta,\tau}|<\varepsilon$,
  \item if $(\sigma,\tau)\in S$, then $\tilde
    X'_{\sigma,\tau}\subseteq\tilde X_{\sigma,\tau}$,
  \item if $\sigma\in\proj[S]$, then $|\dot X'_\sigma,\dot
    X_\sigma|<\varepsilon$,
  \item $|X'_\eta,Y'|<|X_{\eta^*},Y|+\varepsilon$.
  \end{enumerate}
\end{lemma}

\begin{proof}
  Denote $\eta^*$ by $\xi$. Fix large $K<\omega$. Find
  $L<\omega$ and $a^k_{\tau,l}\in X_{\xi,\tau}$ (for
  $l<L,k<K,\tau\in S_\xi$) such that
  \begin{itemize}
  \item $|\{\dot a^k_{\tau,l}:l<L\},\dot
    X_{\xi,\tau}|<\varepsilon$ for each $k,\tau$ (use
    compactness of $\dot H$),
  \item all $\ddot a^k_{\tau,l}$ are distinct and at
    distance $>\delta$ for some fixed $\delta>0$.
  \end{itemize}
  For $l<L$ find a solid cylinder $\tilde
  X^k_{\xi,\tau,l}\subseteq\tilde X_{\xi,\tau}$ of diameter
  less than $\varepsilon$ such that
    $$X^k_{\xi,\tau,l}\subsetnwd \{x\in X_{\xi,\tau}: |\ddot
    x,\ddot a^k_{\tau,l}|<\delta/4\}.$$ Next, find
    $L^k_\tau\subseteq L$ such that $$|\{\dot
    a^k_{\tau,l}:l\in L^k_\tau\},\dot Y'|<|\{\dot
    a^k_{\tau,l}:l\in L\},\dot Y|+\varepsilon$$ ($\dot
    Y'\subseteq \cl\dot Y$ is used here). Now unfold
    $$X^k_{\xi,\tau}=\{x\in X_{\xi,\tau}:\forall l<L\ 
    |\ddot x,\ddot a^k_{\tau,l}|>\delta/3\}$$ to a solid
    cylinder $\tilde X^k_{\xi,\tau}\subseteq \tilde
    X_{\xi,\tau}$.

    For $(\sigma,\tau)\in S$ such that $\xi\subsetneq\sigma$
    find a solid cylinder $\tilde
    X^k_{\sigma,\tau}\subseteq\tilde X_{\sigma,\tau}$ such
    that $$\dot X^k_{\sigma,\tau}=\int_{\dot
      X_{\sigma,\tau}}\cl\dot X^k_{\sigma^*,\tau^*}.$$

    \begin{claim}
      If $(\sigma,\tau)\in S$, $\xi,\subseteq\sigma$ and
      $k'\not=k$, then $$\cl\dot X_{\sigma,\tau}=\cl\dot
      X^k_{\sigma,\tau}\cup\cl\dot X^{k'}_{\sigma,\tau}.$$
    \end{claim}

    \begin{proof}[Proof of Claim]
      Clearly $\dot X_{\xi,\tau}=\dot X_{\xi,\tau}^k\cup\dot
      X_{\xi,\tau}^{k'}$. Using $$\cl\dot X\cup\cl\dot
      Y=\cl\,\int\,\cl\dot X\cup\cl\,\int\,\cl\dot Y$$ we
      get
      \begin{eqnarray*}
        \cl\dot X_{\sigma,\tau}^k\cup\cl\dot
        X^{k'}_{\sigma,\tau}=\cl\,\int_{\dot
          X_{\sigma,\tau}}\,\cl\dot
        X^k_{\sigma^*,\tau^*}\ \cup\ \cl\,\int_{\dot
          X_{\sigma,\tau}}\,\cl\dot X^{k'}_{\sigma^*,\tau^*}\\=\cl\big[\cl_{\dot
          X_{\sigma,\tau}}\,\int_{\dot
          X_{\sigma,\tau}}\,\cl\dot
        X^k_{\sigma^*,\tau^*}\ \cup\ \cl_{\dot
          X_{\sigma,\tau}}\,\int_{\dot
          X_{\sigma,\tau}}\,\cl\dot
        X^{k'}_{\sigma^*,\tau^*}\big]\\=\cl\big[\cl_{\dot
          X_{\sigma,\tau}}\dot
        X^k_{\sigma^*,\tau^*}\ \cup\ \cl_{\dot
          X_{\sigma,\tau}}\dot
        X^{k'}_{\sigma^*,\tau^*}\big]=\cl\dot X_{\sigma^*,\tau^*}^k\cup\cl\dot
        X^{k'}_{\sigma^*,\tau^*}
      \end{eqnarray*}
    \end{proof}
    Now, using Lemma \ref{lemma2} we find $k\in K$ such that
    $$\forall(\sigma,\tau)\in S\quad  \xi\subseteq\sigma\quad\Rightarrow\quad
    |\dot X^k_{\sigma,\tau},\dot
    X_{\sigma,\tau}|<\varepsilon.$$ Let
    $S'=S\cup\{(\eta,\tau^\smallfrown l):l\in L^k_\tau,
    \tau\in S_\xi\}$ and put $$\tilde
    X'_{\eta,\tau^\smallfrown l}=\tilde X^k_{\xi,\tau,l}$$
    (for $l\in L^k_\tau$), for $(\sigma,\tau)\in S$ put
\begin{displaymath}
  \tilde X'_{\sigma,\tau} = \left\{ 
    \begin{array}{ll}
      \tilde X^k_{\sigma,\tau}& \mbox{if}\quad\xi\subseteq\sigma,\\
      \tilde X_{\sigma,\tau} & \mbox{if}\quad \xi\not\subseteq\sigma.
    \end{array} \right.
\end{displaymath}

\end{proof}

\subsection{Construction of $\dot\varphi$ and
  $\ddot\varphi$}\label{sec:Constr}

In order to define $\dot\varphi$ and $\ddot\varphi$ we
shall construct
\begin{itemize}
\item a tree $T$ on $\omega\times\omega$ such that
  $\proj[T]=\btree$ and $T_\sigma$ is finite for each
  $\sigma\in\btree$,
\item a Lusin scheme $\langle\tilde
  Z_{\sigma,\tau}:(\sigma,\tau)\in T\rangle$ with the
  vanishing diameter property of solid subcylinders of
  $\tilde A$ so that
  \begin{itemize}
  \item[(i)] the scheme $\langle\dot
    Z_{\sigma}:\sigma\in\proj[T]\rangle$ also has the
    vanishing diameter property,
  \item[(ii)] $\ddot Z_{\sigma,\tau}$ is relatively open in
    $\bigcup\{\ddot Z_\rho:\lh\rho=\lh\sigma\}$,
  \item[(iii)] $\forall\sigma\in\omega^n\
    \exists\varepsilon_\sigma>0\ \forall\rho\in\omega^n$ $$
    |\rho,\sigma|<\varepsilon_\sigma\quad\Rightarrow\quad|\dot
    Z_\rho,\dot Z_\sigma|<|\rho,\sigma|+2\cdot 2^{-n}$$
  \end{itemize}
\end{itemize}

Suppose we have done this and let $\Phi:\lim
T\rightarrow\tilde A$,
$$\{\Phi(s,t)\}=\bigcap_{n<\omega}\tilde Z_{s\upharpoonright
  n, t\upharpoonright n}$$ be the associated map.

Note that for each $s\in\baire$ the tree $T_s$ is finitely
branching, so $\lim T_s\not=\emptyset$ and by $|\dot
Z_{s\upharpoonright n}|\rightarrow0$ there is exactly one
branch $\psi(s)\in \lim T_s$. By Lemma \ref{lemma3} the map
$\psi:\baire\rightarrow\baire$ is continuous and therefore
its graph is homeomorphic to $\baire$ via
$\bar\psi:\baire\rightarrow\lim T$,
$\bar\psi(s)=(s,\psi(s))$. Define
$\bar\Phi=\Phi\circ\bar\psi$.

Now, $\bar\Phi$ is continuous and $\pi\circ\bar\Phi$ maps
$\baire$ into $A$. Define $\dot\varphi=\dot\pi\circ\bar\Phi$
and $\ddot\varphi=\ddot\pi\circ\bar\Phi$. Note that
$A\circ\dot\varphi=\ddot\varphi=P\circ\ddot\varphi$ since
$\pi\circ\bar\Phi\subseteq A$.

\begin{claim}
  The map $\dot\varphi$ is a homeomorphic embedding from
  $\baire$ with the Cantor topology into $\dot A$.
\end{claim}

\begin{proof}
  Note that the map $\dot\varphi$ is associated with the
  Lusin scheme $\langle \dot
  Z_\sigma:\sigma\in\btree\rangle$. Since the Cantor
  topology is compact, we only need to prove that
  $\dot\varphi$ is continuous. Fix $s\in\baire$ and large
  $n\in\omega$. Consider $r\in\baire$ such that
  $|r,s|<\varepsilon_{s\upharpoonright n}$. Then
  \begin{eqnarray*}
    |\dot\varphi(r),\dot\varphi(s)|\leq |\dot
    Z_{r\upharpoonright n},\dot Z_{\upharpoonright n}|+|\dot
    Z_{s\upharpoonright n}|<|r\har n, s\har
    n|+2\cdot2^{-n}+|\dot Z_{s\upharpoonright n}|\\\leq
    |r,s|+2\cdot 2^{-n}+|\dot Z_{s\upharpoonright n}|
  \end{eqnarray*}
\end{proof}

\begin{claim}
  The map $\ddot\varphi$ is a homeomorphic embedding from
  $\baire$ with the Baire topology into $\ddot A$.
\end{claim}

\begin{proof}
  From the previous claim we get that $\ddot\varphi$ is
  continuous, since the Cantor topology extends the Baire
  topology. To see that it is open note that $\ddot\varphi$
  is the map associated with the Lusin scheme $\langle \ddot
  Z_\sigma:\sigma\in\btree\rangle$ and use (ii).
\end{proof}

\subsection{Construction of the solid tree}

The following lemma is pretty straightforward, so we leave
it without proof.

\begin{lemma}\label{lemma6}
  Let $S$ be a tree and $\langle\ddot X_\sigma: \sigma\in
  S\rangle$ be a Lusin scheme such that for each $\sigma\in
  S$ the set $\ddot X_\sigma$ is relatively open in
  $\bigcup\{\ddot X_\rho: \lh\rho=\lh\sigma\}$. Let
  $\Sigma\subseteq S$ be an antichain. Then for each
  $\sigma\in\Sigma$ the set $\ddot X_\sigma$ is relatively
  open in $\bigcup\{\ddot X_\rho: \rho\in\Sigma\}$.
\end{lemma}

Now, we inductively construct the $i$-th approximation
$T^i=\{(\sigma,\tau)\in T: \#\sigma\leq i\}$ of $T$ and a
solid tree $\langle \tilde
Z^i_{\sigma,\tau}:(\sigma,\tau)\in T^i\rangle$ such that
$\tilde Z^{i+1}_{\sigma,\tau}\subseteq\tilde
Z^i_{\sigma,\tau}$ and $|\tilde
Z^i_{\sigma,\tau}|<2^{-\lh\sigma}$

We ensure that
\begin{itemize}
\item[(1)] if $\sigma\not=\sigma'$, then $|\dot
  Z^i_\sigma,\dot Z^i_{\sigma'}|<|\sigma,\sigma'|$ (hence
  $|\dot Z^i_\sigma,\dot Z^i_{\sigma^k}|<|\sigma,\sigma^k|$
  for each $k$),
\item[(2)] if $\sigma=\sigma'$, then $|\dot
  Z^i_\sigma|<2^{-\lh\sigma}$,
\item[(3)] $|\dot Z^{i+1}_\sigma,\dot Z^i_\sigma|<2^{-i}$
  (hence $|\dot Z^j_\sigma,\dot Z^i_\sigma|<2\cdot
  2^{-\lh\sigma}$ for each $j\geq i$).
\end{itemize}

Once this is done, set $\tilde Z_{\sigma,\tau}=\tilde
Z^{\#\sigma}_{\sigma,\tau}$ and note that
\begin{itemize}
\item[(4)] $|\dot Z_\sigma,\dot Z_{\sigma^k}|\leq
  |\sigma,\sigma^k|+2\cdot 2^{-\lh\sigma}$. 
\end{itemize}
Indeed, $|\dot Z_\sigma,\dot Z_{\sigma^k}|\leq |\dot
Z^{\#\sigma}_\sigma,\dot Z^{\#\sigma}_{\sigma^k}|+|\dot
Z^{\#\sigma}_{\sigma^k},\dot
Z^{\#\sigma^k}_{\sigma^k}|<|\sigma,\sigma^k|+2\cdot
2^{-\lh\sigma}$.

Moreover,
\begin{itemize}
\item[(5)] $\forall n<\omega\ \forall\sigma\in\omega^n\quad
  |\dot Z_\sigma|<5\cdot 2^{-n}$.
\end{itemize}
Indeed, by (2) and (4) we get $ |\dot Z_\sigma|\leq
2\cdot|\dot Z^{\#\sigma}_\sigma,\dot
Z^{\#\sigma}_{\sigma^n}|+|\dot Z^{\#\sigma}_{\sigma^n}|
<2\cdot|\sigma,\sigma^n|+2^{-n}<5\cdot 2^{-n}$.

Now we show how the above construction is used in Section
\ref{sec:Constr}. Note that (5) implies (i). To see (ii) and
(iii) observe that:
\begin{itemize}
\item (ii) follows by Lemma \ref{lemma6} applied to the
  Lusin scheme, whose $i$-th level is $\{\ddot
  Z^i_{\sigma,\tau}:\#\sigma\leq i\}$. The set
  $\{((\sigma,\tau),i)): (\sigma,\tau)\in T \wedge
  \#\sigma\leq i\}$ is given a tree ordering so that
  $((\sigma,\tau),i)\leq ((\sigma,\tau),i+1)$ and
  $((\sigma^*,\tau^*),i)\leq((\sigma,\tau),i+1)$; so
  \begin{displaymath}
    ((\sigma,\tau),i+1)^* = \left\{ 
      \begin{array}{ll}
        ((\sigma,\tau),i)& \mbox{if}\quad \#\sigma\leq i,\\
        ((\sigma^*,\tau^*),i) & \mbox{if}\quad \#\sigma=i+1.
      \end{array} \right.
  \end{displaymath}
\item to see (iii), fix $\sigma\in\omega^n$ and choose
  $\varepsilon_\sigma$ so small that if
  $|\rho,\sigma|<\varepsilon_\sigma$, then for each $m<n$:
  \begin{itemize}
  \item[] if $\sigma(m)\not=0$, then $\sigma(m)=\rho(m)$,
  \item[] if $\sigma(m)=0$, then $\rho(m)=0$ or
    $\rho(m)>\max\sigma, n$.  
  \end{itemize}
  Then $|\rho,\sigma|<\varepsilon_\sigma$ implies that
  $\sigma=\rho^k$ for some $k\leq n$ and thus
  $|Z_\rho,Z_\sigma|<|\rho,\sigma|+2\cdot 2^{-n}$, by (4).
\end{itemize}

\bigskip

\noindent The construction of the trees $T^i$ goes as follows.

\smallskip

\noindent \textbf{Step 0}. Put $T^0=\{\emptyset,\emptyset\}$
and $\tilde Z^0_{\emptyset,\emptyset}=\tilde A$.

\noindent \textbf{Step $\mathbf{n\rightarrow n+1}$}. 

Apply Lemma \ref{lemma5} to small $\varepsilon>0$, $S=T^n$,
$\eta$ such that $\#\eta=n+1$ and $\tilde
X_{\sigma,\tau}=\tilde Z^n_{\sigma,\tau}$ for
$(\sigma,\tau)\in T^n$, as $Y$ and $Y'$ use:
\begin{itemize}
\item[(\textbf{A})] if $\eta\not=\eta'$, then use
  $Y=Z^n_{\eta'^*}$ and $Y'=Z^n_{\eta'}$,
\item[(\textbf{B})] if $\eta=\eta'$, then use
  $Y=Z^n_{\eta^*}$ and $Y'=\{y\}$ for any $y\in Y$,
\end{itemize}
(note that in both cases we have $\dot Y'\subseteq\cl\dot
Y$). 

Put $T^{n+1}=S'$ and $\tilde Z^{n+1}_{\sigma,\tau}=\tilde
X'_{\sigma,\tau}$ for $(\sigma,\tau)\in T^{n+1}$.

We need to verify $(1)$ and $(2)$, small $\varepsilon$ takes
care of $(3)$. Pick $\sigma$. There are two cases

\noindent \textbf{Case 1}.  Suppose $\sigma=\eta$.

\noindent \textbf{Subcase 1A}. If we are in case $(A)$ of
the construction, then
$$|\dot Z^{n+1}_\eta,\dot Z^{n+1}_{\eta'}|<|\dot
Z^{n+1}_\eta,\dot Z^n_{\eta'}|+\varepsilon<|\dot
Z^n_{\eta^*},\dot Z^n_{\eta'^*}|+2\varepsilon$$ the first
inequality follows from Lemma \ref{lemma5}(3) and the second
follows from Lemma \ref{lemma5}(4) and choice of $Y,Y'$.
Next, if $\eta^*=\eta'^*$, then $|\dot Z^n_{\eta^*},\dot
Z^n_{\eta'^*}|=0$ and we are done; else if
$\eta^*\not=\eta'^*$, then $\eta'^*={\eta^*}'$ (the same
place changes in $\eta$ and $\eta^*$) and
$$|\dot
Z^n_{\eta^*},\dot Z^n_{\eta'^*}|+2\varepsilon\leq|\dot
Z^n_{\eta^*},\dot
Z^n_{{\eta^*}'}|+2\varepsilon<|\eta^*,{\eta^*}'|=|\eta,\eta'|$$
by small $\varepsilon$ and induction hypothesis.

\noindent \textbf{Subcase 1B}. If we are in case $(B)$ of
the construction, then we have
$$|\dot Z^{n+1}_\eta|\leq 2\cdot|\dot
Z^{n+1}_\eta,Y'|<2\cdot |\dot
Z^n_{\eta^*},Y|+2\varepsilon=2\varepsilon$$ since $|\dot
Z^n_{\eta^*},Y|=0$.

\noindent \textbf{Case 2}. Suppose $\sigma\not=\eta$. Now
$(1)$ follows by $$|\dot
Z^{n+1}_\sigma,\dot Z^{n+1}_{\sigma'}|<|\dot
Z^{n+1}_\sigma,\dot Z^n_\sigma|+|\dot Z^n_\sigma,\dot
Z^n_{\sigma'}|+|\dot Z^n_{\sigma'},\dot
Z^{n+1}_{\sigma'}|<\varepsilon+|\dot Z^n_\sigma,\dot
Z^n_{\sigma'}|+\varepsilon<|\sigma,\sigma'|$$ by the
induction hypothesis and since $\varepsilon$ is small
enough.

Finally, $(2)$ follows by $Z^{n+1}_\sigma\subseteq
Z^n_\sigma$.

\end{document}